\documentclass[12pt]{article}
\usepackage{amssymb}
\usepackage{amsthm}
\usepackage{amsmath}
\usepackage{graphicx}
\usepackage{enumerate}

\DeclareMathOperator{\Con}{Con}

\newtheorem{theorem}{Theorem}

\newtheorem{example}[theorem]{Example}
\newtheorem{corollary}[theorem]{Corollary}
\title{Varieties corresponding to classes of complemented posets}
\author{Ivan~Chajda, Miroslav~Kola\v r\'ik and Helmut~L\"anger}
\date{}
\begin{document}
\footnotetext[1]{Support of the research of the first and third author by \"OAD, project CZ~02/2019, and support of the research of the first author by IGA, project P\v rF~2019~015, is gratefully acknowledged.}
\maketitle
\begin{abstract}
As algebraic semantics of the logic of quantum mechanics there are usually used orthomodular posets, i.e.\ bounded posets with a complementation which is an antitone involution and where the join of orthogonal elements exists and the orthomodular law is satisfied. When we omit the condition that the complementation is an antitone involution, then we obtain skew-orthomodular posets. To each such poset we can assign a bounded $\lambda$-lattice in a non-unique way. Bounded $\lambda$-lattices are lattice-like algebras whose operations are not necessarily associative. We prove that any of the following properties for bounded posets with a unary operation can be characterized by certain identities of an arbitrary assigned $\lambda$-lattice: complementarity, orthogonality, almost skew-orthomodularity and skew-orthomodularity. It is shown that these identities are independent. Finally, we show that the variety of skew-orthomodular $\lambda$-lattices is congruence permutable as well as congruence regular.
\end{abstract}



It is well-known that an algebraic semantics of the logic of quantum mechanics is provided by means of orthomodular lattices as shown by G.~Birkhoff and J.~von~Neumann (\cite{BV}) or, independently, by K.~Husimi (\cite H). The details of this construction can be found e.g.\ in the monograph by L.~Beran (\cite B). However, it was shown later that in the logic of quantum mechanics the connective disjunction represented by the lattice operation $\vee$ need not exist for elements that are not orthogonal. Hence the concept of an orthomodular poset was introduced as follows:

A bounded {\em poset} $\mathbf P=(P,\leq,{}',0,1)$ with a unary operation is called {\em orthomodular} (see e.g.\ \cite{BH}) if $'$ is an antitone involution on $(P,\leq)$ which is a complementation, i.e.\ $x\leq y$ implies $y'\leq x'$, $x''=x$, $\sup(x,x')$ exists for all $x\in P$ and it is equal to $1$, and $\inf(x,x')$ exists for all $x\in P$ and it is equal to $0$; moreover, $\sup(x,y)$ must exist in case $x\leq y'$; finally, for all $x,y\in P$ with $x\leq y$ there exists $x\vee(x'\wedge y)$ and it is equal to $y$. Due to De Morgan's laws, also dually, $y\wedge(y'\vee x)$ exists for all $x,y\in P$ with $x\leq y$ and it is equal to $x$. The last property is called the {\em orthomodular law} and can be expressed in the case of lattices alternatively in the form of the equivalent identities
\begin{align*}
  x\vee(x'\wedge(x\vee y)) & \approx x\vee y, \\
y\wedge(y'\vee(x\wedge y)) & \approx x\wedge y.
\end{align*}
It was shown by V.~Sn\'a\v sel (\cite S) that every bounded poset can be organized into a so-called bounded $\lambda$-lattice. Bounded $\lambda$-lattices can be considered as bounded lattices whose binary operations are not necessarily associative. More precisely, a bounded $\lambda$-lattice is a bounded lattice if and only if its binary operations are associative, see \cite{CL11} for details. The notion of a $\lambda$-lattice was successfully used by the first two authors for constructing a variety of $\lambda$-lattices which corresponds to the class of orthomodular posets. This variety turns out to be congruence permutable and congruence regular. Of course, it is of advantage to work with varieties of algebras instead of classes of posets since for varieties the well-known methods of Universal Algebra can be applied.

Back to the logic of quantum mechanics, we take as an appropriate structure for the algebraic semantics complemented posets in which the join of two orthogonal elements exists and which satisfy the orthomodular law. We do not ask this complementation to be an antitone involution. The concept of complemented lattices which satisfy the orthomodular law, but whose complementation need not be an antitone involution was introduced in \cite{BC} and studied in \cite{CL18b}. This concept was generalized by the first and third author to posets in \cite{CL18a}. In the present paper we show that similarly to the case of orthomodular posets, for the posets described above the method of considering assigned $\lambda$-lattices can be successfully applied. In fact, it turns out that important properties of certain posets can be characterized by identities of assigned $\lambda$-lattices.

Let $\mathbf P=(P,\leq,{}',0,1)$ be a bounded poset with a unary operation and $a,b\in P$. We define
\begin{align*}
L(a,b) & :=\{x\in P\mid x\leq a,b\}, \\
U(a,b) & :=\{x\in P\mid a,b\leq x\}.
\end{align*}
If there exists $\sup(a,b)$ or $\inf(a,b)$ then we will denote these elements by $a\vee b$ or $a\wedge b$, respectively.

We call $\mathbf P$ {\em complemented} if it satisfies the identities $x\vee x'\approx1$ and $x\wedge x'\approx0$. In this case the operation $'$ is called a {\em complementation}. We say that $a,b$ are {\em orthogonal elements} of $P$, shortly $a\perp b$, if $a\leq b'$. We call a {\em complemented poset} $\mathbf P$ {\em orthogonal} if $x\vee y$ exists for arbitrary orthogonal elements $x,y$ of $P$. We call an orthogonal poset $\mathbf P$ {\em almost skew-orthomodular} if $x\vee(x'\wedge y)$ exists for all $x,y\in P$ with $x\leq y$. We call an almost skew-orthomodular poset $\mathbf P$ {\em skew-orthomodular} if $x\vee(x'\wedge y)=y$ for all $x,y\in P$ with $x\leq y$.

\begin{example}
The poset shown in Fig.~1

\vspace*{-4mm}

\[
\setlength{\unitlength}{7mm}
\begin{picture}(12,9)
\put(6,2){\circle*{.3}}
\put(1,4){\circle*{.3}}
\put(3,4){\circle*{.3}}
\put(5,4){\circle*{.3}}
\put(7,4){\circle*{.3}}
\put(9,4){\circle*{.3}}
\put(11,4){\circle*{.3}}
\put(3,6){\circle*{.3}}
\put(5,6){\circle*{.3}}
\put(7,6){\circle*{.3}}
\put(9,6){\circle*{.3}}
\put(6,8){\circle*{.3}}
\put(6,2){\line(-5,2)5}
\put(6,2){\line(-3,2)3}
\put(6,2){\line(-1,2)1}
\put(6,2){\line(1,2)1}
\put(6,2){\line(3,2)3}
\put(6,2){\line(5,2)5}
\put(1,4){\line(1,1)2}
\put(1,4){\line(2,1)4}
\put(3,4){\line(0,1)2}
\put(3,4){\line(1,1)2}
\put(5,4){\line(-1,1)2}
\put(5,4){\line(0,1)2}
\put(5,4){\line(1,1)2}
\put(7,4){\line(-2,1)4}
\put(7,4){\line(1,1)2}
\put(9,4){\line(-2,1)4}
\put(9,4){\line(0,1)2}
\put(11,4){\line(-2,1)4}
\put(11,4){\line(-1,1)2}
\put(6,8){\line(-3,-2)3}
\put(6,8){\line(-1,-2)1}
\put(6,8){\line(1,-2)1}
\put(6,8){\line(3,-2)3}
\put(5.875,1.25){$0$}
\put(.35,3.85){$a$}
\put(2.35,3.85){$b$}
\put(4.35,3.85){$c$}
\put(7.3,3.85){$d$}
\put(9.3,3.85){$e$}
\put(11.3,3.85){$f$}
\put(2.35,5.85){$g$}
\put(4.35,5.85){$h$}
\put(7.3,5.85){$i$}
\put(9.3,5.85){$j$}
\put(5.85,8.4){$1$}
\put(5.2,.3){{\rm Fig.~1}}
\end{picture}
\]

\vspace*{-3mm}

with
\[
\begin{array}{c|cccccccccccc}
x  & 0 & a & b & c & d & e & f & g & h & i & j & 1 \\
\hline
x' & 1 & j & j & j & i & i & h & f & f & e & c & 0
\end{array}
\]
belongs to ${\mathcal Ps}$, but is not a lattice. Moreover, $'$ is antitone, but not an involution.
\end{example}

In the following let ${\mathcal Pc}$, ${\mathcal Po}$, ${\mathcal Pa}$ and ${\mathcal Ps}$ denote the class of all complemented, orthogonal, almost skew-orthomodular and skew-orthomodular posets, respectively. We are going to show that these classes do not coincide, i.e.\ the inclusions are proper.

\begin{theorem}
We have
\[
{\mathcal Ps}\subsetneqq{\mathcal Pa}\subsetneqq{\mathcal Po}\subsetneqq{\mathcal Pc}.
\]
\end{theorem}

\begin{proof}
The poset shown in Fig.~2

\vspace*{-4mm}

\[
\setlength{\unitlength}{7mm}
\begin{picture}(6,9)
\put(3,2){\circle*{.3}}
\put(1,4){\circle*{.3}}
\put(3,4){\circle*{.3}}
\put(5,4){\circle*{.3}}
\put(1,6){\circle*{.3}}
\put(3,6){\circle*{.3}}
\put(5,6){\circle*{.3}}
\put(3,8){\circle*{.3}}
\put(3,2){\line(-1,1)2}
\put(3,2){\line(0,1)6}
\put(3,2){\line(1,1)2}
\put(1,4){\line(0,1)2}
\put(1,4){\line(1,1)2}
\put(3,4){\line(-1,1)2}
\put(5,4){\line(0,1)2}
\put(3,8){\line(-1,-1)2}
\put(3,8){\line(1,-1)2}
\put(2.875,1.25){$0$}
\put(.35,3.85){$a$}
\put(3.3,3.85){$b$}
\put(5.3,3.85){$c$}
\put(.35,5.85){$d$}
\put(3.3,5.85){$e$}
\put(5.3,5.85){$f$}
\put(2.85,8.4){$1$}
\put(2.2,.3){{\rm Fig.~2}}
\end{picture}
\]

\vspace*{-3mm}

with
\[
\begin{array}{c|cccccccc}
x  & 0 & a & b & c & d & e & f & 1 \\
\hline
x' & 1 & c & c & d & f & f & a & 0
\end{array}
\]
is not a lattice and belongs to ${\mathcal Pa}\setminus{\mathcal Ps}$ since $a\leq d$, but $a\vee(a'\wedge d)=a\vee(c\wedge d)=a\vee0=a\neq d$. Moreover, $'$ is neither antitone nor an involution. The poset shown in Fig.~3

\vspace*{-4mm}

\[
\setlength{\unitlength}{7mm}
\begin{picture}(8,9)
\put(4,2){\circle*{.3}}
\put(3,4){\circle*{.3}}
\put(5,4){\circle*{.3}}
\put(7,5){\circle*{.3}}
\put(1,4){\circle*{.3}}
\put(3,6){\circle*{.3}}
\put(5,6){\circle*{.3}}
\put(4,8){\circle*{.3}}
\put(4,2){\line(-3,2)3}
\put(4,2){\line(-1,2)1}
\put(4,2){\line(1,1)3}
\put(4,2){\line(1,2)1}
\put(1,4){\line(1,1)2}
\put(3,4){\line(0,1)2}
\put(3,4){\line(1,1)2}
\put(5,4){\line(-1,1)2}
\put(5,4){\line(0,1)2}
\put(4,8){\line(-1,-2)1}
\put(4,8){\line(1,-2)1}
\put(4,8){\line(1,-1)3}
\put(3.875,1.25){$0$}
\put(.35,3.85){$a$}
\put(2.35,3.85){$b$}
\put(5.3,3.85){$c$}
\put(7.3,4.85){$d$}
\put(2.35,5.85){$e$}
\put(5.3,5.85){$f$}
\put(3.85,8.4){$1$}
\put(3.2,.3){{\rm Fig.~3}}
\end{picture}
\]

\vspace*{-3mm}

with
\[
\begin{array}{c|cccccccc}
x  & 0 & a & b & c & d & e & f & 1 \\
\hline
x' & 1 & f & d & d & a & d & d & 0
\end{array}
\]
belongs to ${\mathcal Po}\setminus{\mathcal Pa}$ since $a\leq e$, but $a'\wedge e=f\wedge e$ does not exist. Moreover, $'$ is neither antitone nor an involution. The poset shown in Fig.~4

\vspace*{-4mm}

\[
\setlength{\unitlength}{7mm}
\begin{picture}(8,9)
\put(4,2){\circle*{.3}}
\put(3,4){\circle*{.3}}
\put(5,4){\circle*{.3}}
\put(7,5){\circle*{.3}}
\put(1,6){\circle*{.3}}
\put(3,6){\circle*{.3}}
\put(5,6){\circle*{.3}}
\put(4,8){\circle*{.3}}
\put(4,2){\line(-1,2)1}
\put(4,2){\line(1,1)3}
\put(4,2){\line(1,2)1}
\put(3,4){\line(-1,1)2}
\put(3,4){\line(0,1)2}
\put(3,4){\line(1,1)2}
\put(5,4){\line(-1,1)2}
\put(5,4){\line(0,1)2}
\put(4,8){\line(-3,-2)3}
\put(4,8){\line(-1,-2)1}
\put(4,8){\line(1,-2)1}
\put(4,8){\line(1,-1)3}
\put(3.875,1.25){$0$}
\put(2.35,3.85){$a$}
\put(5.3,3.85){$b$}
\put(7.3,4.85){$c$}
\put(.35,5.85){$d$}
\put(2.35,5.85){$e$}
\put(5.3,5.85){$f$}
\put(3.85,8.4){$1$}
\put(3.2,.3){{\rm Fig.~4}}
\end{picture}
\]

\vspace*{-3mm}

with
\[
\begin{array}{c|cccccccc}
x  & 0 & a & b & c & d & e & f & 1 \\
\hline
x' & 1 & c & d & a & c & c & c & 0
\end{array}
\]
belongs to ${\mathcal Pc}\setminus{\mathcal Po}$ since $a\leq d=b'$, but $a\vee b$ does not exist. Moreover, $'$ is neither antitone nor an involution.
\end{proof}

Now we introduce the concept of a bounded $\lambda$-lattice taken from \cite S.

A {\em bounded $\lambda$-lattice} is an algebra $(L,\sqcup,\sqcap,0,1)$ of type $(2,2,0,0)$ satisfying the identities
\begin{align*}
                   x\sqcup y & \approx y\sqcup x,\;x\sqcap y\approx y\sqcap x, \\
x\sqcup((x\sqcup y)\sqcup z) & \approx(x\sqcup y)\sqcup z,\;x\sqcap((x\sqcap y)\sqcap z)\approx(x\sqcap y)\sqcap z, \\
          x\sqcup(x\sqcap y) & \approx x,\;x\sqcap(x\sqcup y)\approx x, \\
          					x\sqcup0 & \approx x,\;x\sqcup1\approx1.
\end{align*}
Hence the class of bounded $\lambda$-lattices forms a variety. Notice that every bounded $\lambda$-lattice satisfies the identities
\[
x\sqcap0\approx0\text{ and }x\sqcap1\approx x.
\]
Recall from \cite{CL11} that every variety of bounded $\lambda$-lattices is congruence distributive. It is well-known that in every bounded $\lambda$-lattice $x\sqcup y=y$ is equivalent to $x\sqcap y=x$.

Let $\mathbf P=(P,\leq,{}',0,1)$ be a bounded poset with a unary operation. We introduce binary operations $\sqcup$ and $\sqcap$ on $P$ as follows ($x,y\in P$): If $x\vee y$ exists then $x\sqcup y:=x\vee y$. Otherwise $x\sqcup y=y\sqcup x$ is an arbitrary element of $U(x,y)$. If $x\wedge y$ exists then $x\sqcap y:=x\wedge y$. Otherwise $x\sqcap y=y\sqcap x$ is an arbitrary element of $L(x,y)$. Then $(P,\sqcup,\sqcap,{}',0,1)$ is a bounded $\lambda$-lattice with a unary operation which we call {\em $\lambda$-lattice assigned} to the bounded poset $\mathbf P$. Let $\mathbb A(\mathbf P)$ denote the set of all $\lambda$-lattices assigned to $\mathbf P$.

To every bounded $\lambda$-lattice $\mathbf L=(L,\sqcup,\sqcap,{}',0,1)$ with a unary operation we assign a bounded poset $\mathbb P(\mathbf L)=(L,\leq,{}',0,1)$ as follows:
\[
x\leq y\text{ if and only if }x\sqcup y=y
\]
($x,y\in L$). It was shown in \cite S that $(L,\leq,0,1)$ is a bounded poset and
\[
x\leq y\text{ if and only if }x\sqcap y=x.
\]
Moreover, using the absorption laws, we easily derive the identities
\[
x\sqcup x\approx x\text{ and }x\sqcap x\approx x.
\]
For $i\in\{c,o,a,s\}$ let ${\mathcal Li}$ denote the class of all bounded $\lambda$-lattices $\mathbf L$ with a unary operation satisfying $\mathbb P(\mathbf L)\in{\mathcal Pi}$. Hence ${\mathcal Li}$ can be considered as a representation of ${\mathcal Pi}$. This means that the properties of ${\mathcal Li}$ may be considered as properties of ${\mathcal Pi}$.

In the following we will characterize the above mentioned properties of bounded posets with a unary operation by means of identities of assigned $\lambda$-lattices. Surprisingly, this works despite the fact that this assignment is not unique. Hence, classes of complemented, orthogonal, almost skew-orthomodular and skew-orthomodular posets will be characterized by means of varieties of bounded $\lambda$-lattices. We start with complemented posets.

\begin{theorem}
Let $\mathbf P=(P,\leq,{}',0,1)$ be a bounded poset with a unary operation and $\mathbf L=(P,\sqcup,\sqcap,{}',0,1)\in\mathbb A(\mathbf P)$. Then $\mathbf P$ is complemented if and only if $\mathbf L$ satisfies the identities
\begin{eqnarray}
(x\sqcup y)\sqcup(x'\sqcup y) & \approx & 1,\label{equ1} \\
(x\sqcap y)\sqcap(x'\sqcap y) & \approx & 0.\label{equ2}
\end{eqnarray}
Hence ${\mathcal Lc}$ is a variety.
\end{theorem}

\begin{proof}
Let $a,b\in P$. First assume $\mathbf P\in{\mathcal Pc}$. Then
\begin{align*}
 a & \leq a\sqcup b\leq(a\sqcup b)\sqcup(a'\sqcup b), \\
a' & \leq a'\sqcup b\leq(a\sqcup b)\sqcup(a'\sqcup b)
\end{align*}
and hence $(a\sqcup b)\sqcup(a'\sqcup b)\in U(a,a')=\{1\}$, i.e., $(a\sqcup b)\sqcup(a'\sqcup b)=1$. Dually,
\begin{align*}
(a\sqcap b)\sqcap(a'\sqcap b) & \leq a\sqcap b\leq a, \\
(a\sqcap b)\sqcap(a'\sqcap b) & \leq a'\sqcap b\leq a'
\end{align*}
and hence $(a\sqcap b)\sqcap(a'\sqcap b)\in L(a,a')=\{0\}$, i.e., $(a\sqcap b)\sqcap(a'\sqcap b)=0$. Conversely, suppose $\mathbf L$ to satisfy identities (\ref{equ1}) and (\ref{equ2}). If $a,a'\leq b$ then $a\sqcup b=a'\sqcup b=b$ and hence
\[
b=b\sqcup b=(a\sqcup b)\sqcup(a'\sqcup b)=1
\]
showing $a\vee a'=1$. Similarly, $b\leq a,a'$ implies $a\sqcap b=a'\sqcap b=b$ and therefore
\[
b=b\sqcap b=(a\sqcap b)\sqcap(a'\sqcap b)=0
\]
showing $a\wedge a'=0$. Hence $\mathbf P\in{\mathcal Pc}$.
\end{proof}

The identities
\begin{align*}
x\sqcup x' & \approx1, \\
x\sqcap x' & \approx0.
\end{align*}
are necessary, but not sufficient for a bounded $\lambda$-lattice to be complemented.

Orthogonal posets can be characterized by an identity that is a bit more complicated than the previous ones (\ref{equ1}) and (\ref{equ2}).

\begin{theorem}\label{th2}
Let $\mathbf P=(P,\leq,{}',0,1)\in{\mathcal Pc}$ and $\mathbf L=(P,\sqcup,\sqcap,{}',0,1)\in\mathbb A(\mathbf P)$. Then $\mathbf P$ is orthogonal if and only if $\mathbf L$ satisfies the identity
\begin{equation}
(((x\sqcap y')\sqcup z)\sqcup(y\sqcup z))\sqcap((x\sqcap y')\sqcup y)\approx(x\sqcap y')\sqcup y.\label{equ3}
\end{equation}
Hence ${\mathcal Lo}$ is a variety.
\end{theorem}

\begin{proof}
Let $a,b,c\in P$. First assume $\mathbf P\in{\mathcal Po}$. Then $(a\sqcap b')\vee b$ exists. Now
\begin{align*}
a\sqcap b' & \leq (a\sqcap b')\sqcup c\leq((a\sqcap b')\sqcup c)\sqcup(b\sqcup c), \\
         b & \leq b\sqcup c\leq((a\sqcap b')\sqcup c)\sqcup(b\sqcup c)
\end{align*}
and hence $((a\sqcap b')\sqcup c)\sqcup(b\sqcup c)\in U(a\sqcap b',b)$ which yields
\[
(a\sqcap b')\sqcup b\leq((a\sqcap b')\sqcup c)\sqcup(b\sqcup c)
\]
which is equivalent to identity (\ref{equ3}). Conversely, suppose $\mathbf L$ to satisfy identity (\ref{equ3}). Assume $a\perp b$. Then $a\sqcap b'=a$. If $a,b\leq c$ then $a\sqcup c=b\sqcup c=c$ and hence
\begin{align*}
a\sqcup b & =(a\sqcap b')\sqcup b=(((a\sqcap b')\sqcup c)\sqcup(b\sqcup c))\sqcap((a\sqcap b')\sqcup b)\leq((a\sqcap b')\sqcup c)\sqcup(b\sqcup c)= \\
          & =(a\sqcup c)\sqcup c=c\sqcup c=c
\end{align*}
showing $a\sqcup b=a\vee b$, i.e.\ $a\vee b$ exists. Hence $\mathbf P\in{\mathcal Po}$.
\end{proof}

Similarly as above we can characterize almost skew-orthomodular posets.

\begin{theorem}\label{th1}
Let $\mathbf P=(P,\leq,{}',0,1)\in{\mathcal Po}$ and $\mathbf L=(P,\sqcup,\sqcap,{}',0,1)\in\mathbb A(\mathbf P)$. Then $\mathbf P$ is almost skew-orthomodular if and only if $\mathbf L$ satisfies the identity
\begin{equation}
(((x\sqcap y)'\sqcap z)\sqcap(y\sqcap z))\sqcup((x\sqcap y)'\sqcap y)\approx(x\sqcap y)'\sqcap y.\label{equ4}
\end{equation}
Hence ${\mathcal La}$ is a variety.
\end{theorem}

\begin{proof}
Let $a,b,c\in P$. First assume $\mathbf P\in{\mathcal Pa}$. Then $(a\sqcap b)'\wedge b$ exists. Now
\begin{align*}
((a\sqcap b)'\sqcap c)\sqcap(b\sqcap c) & \leq(a\sqcap b)'\sqcap c\leq(a\sqcap b)', \\
((a\sqcap b)'\sqcap c)\sqcap(b\sqcap c) & \leq b\sqcap c\leq b
\end{align*}
and hence $((a\sqcap b)'\sqcap c)\sqcap(b\sqcap c)\in L((a\sqcap b)',b)$ which yields
\[
((a\sqcap b)'\sqcap c)\sqcap(b\sqcap c)\leq(a\sqcap b)'\sqcap b
\]
which is equivalent to identity (\ref{equ4}). Conversely, suppose $\mathbf L$ to satisfy identity (\ref{equ4}). Assume $a\leq b$. Then $a\sqcap b=a$. If $c\leq a',b$ then $a'\sqcap c=b\sqcap c=c$ and hence
\begin{align*}
c & =c\sqcap c=(a'\sqcap c)\sqcap c=((a\sqcap b)'\sqcap c)\sqcap(b\sqcap c)\leq \\
  & \leq(((a\sqcap b)'\sqcap c)\sqcap(b\sqcap c))\sqcup((a\sqcap b)'\sqcap b)=(a\sqcap b)'\sqcap b=a'\sqcap b
\end{align*}
showing $a'\sqcap b=a'\wedge b$, i.e.\ $a'\wedge b$ exists. Moreover, $a\sqcup b=b$. If $a,a'\sqcap b\leq c$ then $a\sqcup c=(a'\sqcap b)\sqcup c=c$ and hence
\begin{align*}
a\sqcup(a'\sqcap b) & =(b\sqcap a')\sqcup a=(((b\sqcap a')\sqcup c)\sqcup(a\sqcup c))\sqcap((b\sqcap a')\sqcup a)\leq \\
                    & \leq((b\sqcap a')\sqcup c)\sqcup(a\sqcup c)=((a'\sqcap b)\sqcup c)\sqcup c=c\sqcup c=c
\end{align*}
showing $a\sqcup(a'\sqcap b)=a\vee(a'\wedge b)$, i.e.\ $a\vee(a'\wedge b)$ exists. Hence $\mathbf P\in{\mathcal Pa}$.
\end{proof}

\begin{example}
The poset shown in Fig.~5

\vspace*{-4mm}

\[
\setlength{\unitlength}{7mm}
\begin{picture}(8,9)
\put(3,2){\circle*{.3}}
\put(1,4){\circle*{.3}}
\put(3,4){\circle*{.3}}
\put(5,4){\circle*{.3}}
\put(7,4){\circle*{.3}}
\put(1,6){\circle*{.3}}
\put(3,6){\circle*{.3}}
\put(5,6){\circle*{.3}}
\put(3,8){\circle*{.3}}
\put(3,2){\line(-1,1)2}
\put(3,2){\line(0,1)6}
\put(3,2){\line(1,1)2}
\put(3,2){\line(2,1)4}
\put(1,4){\line(0,1)2}
\put(1,4){\line(1,1)2}
\put(3,4){\line(-1,1)2}
\put(5,4){\line(-2,1)4}
\put(5,4){\line(0,1)2}
\put(7,4){\line(-2,1)4}
\put(3,8){\line(-1,-1)2}
\put(3,8){\line(1,-1)4}
\put(2.875,1.25){$0$}
\put(.35,3.85){$a$}
\put(3.3,3.85){$b$}
\put(5.3,3.85){$c$}
\put(7.3,3.85){$d$}
\put(.35,5.85){$e$}
\put(3.3,5.85){$f$}
\put(5.3,5.85){$g$}
\put(2.85,8.4){$1$}
\put(2.2,.3){{\rm Fig.~5}}
\end{picture}
\]

\vspace*{-3mm}

with
\[
\begin{array}{c|ccccccccc}
x  & 0 & a & b & c & d & e & f & g & 1 \\
\hline
x' & 1 & g & g & 1 & 1 & g & g & 1 & g
\end{array}
\]
satisfies identity {\rm(\ref{equ4})}, but does not belong to ${\mathcal Pc}$ since $c\wedge c'=c\wedge1=c\neq0$, and is not a lattice. Moreover, $'$ is antitone, but not an involution.
\end{example}

Next we characterize skew-orthomodular posets by identities of assigned $\lambda$-lattices. Since in almost skew-orthomodular $\lambda$-lattices we have
\[
x\sqcup(x'\sqcap(x\sqcup y))=x\vee(x'\wedge(x\vee y)),
\]
we only need to add a single identity. Let us note that the poset shown in Fig.~1 is almost skew-orthomodular, but not skew-orthomodular. For skew-orthomodularity we have the following result. The proof is evident.

\begin{corollary}
Let $\mathbf P=(P,\leq,{}',0,1)$ be a bounded poset with a unary operation and $\mathbf L=(P,\sqcup,\sqcap,{}',0,1)\in\mathbb A(\mathbf P)$. Then $\mathbf P$ is skew-orthomodular if and only if $\mathbf L$ satisfies the identities {\rm(\ref{equ1})} -- {\rm(\ref{equ5})} where
\begin{equation}
x\sqcup(x'\sqcap(x\sqcup y))\approx x\sqcup y.\label{equ5}
\end{equation}
Hence ${\mathcal Ls}$ is a variety.
\end{corollary}

In the following we show some important congruence properties of the variety ${\mathcal Ls}$.

Let $\mathcal V$ be a variety. \\
The variety $\mathcal V$ is called {\em congruence permutable} if $\Theta\circ\Phi=\Phi\circ\Theta$ for all $\mathbf A\in\mathcal V$ and all $\Theta,\Phi\in\Con\mathbf A$. \\
The variety $\mathcal V$ is called {\em congruence regular} if for each $\mathbf A=(A,F)\in\mathcal V$, $a\in A$ and $\Theta,\Phi\in\Con\mathbf A$ with $[a]\Theta=[a]\Phi$ we have $\Theta=\Phi$. \\
It is well-known (cf.\ \cite{CEL}) that $\mathcal V$ is congruence permutable if and only if there exists a so-called {\em Malcev term}, i.e.\ a ternary term $p$ satisfying
\[
p(x,x,y)\approx p(y,x,x)\approx y
\]
and it is regular if and only if there exists a positive integer $n$ and ternary terms $t_1,\ldots,t_n$ such that
\[
t_1(x,y,z)=\cdots=t_n(x,y,z)=z\text{ if and only if }x=y.
\]

\begin{theorem}
Let $\mathcal V$ be a variety of bounded $\lambda$-lattices $(L,\sqcup,\sqcap,{}',0,1)$ with a unary operation satisfying the identities $x\sqcap x'\approx0$ and {\rm(5)}. Then $\mathcal V$ is congruence permutable. In particular, ${\mathcal Ls}$ is congruence permutable.
\end{theorem}

\begin{proof}
The term
\[
p(x,y,z):=(x\sqcup(y'\sqcap(y\sqcup z)))\sqcap(z\sqcup(y'\sqcap(y\sqcup x)))
\]
is a Malcev term since
\begin{align*}
p(x,x,z) & \approx(x\sqcup(x'\sqcap(x\sqcup z)))\sqcap(z\sqcup(x'\sqcap(x\sqcup x)))\approx(x\sqcup z)\sqcap(z\sqcup(x'\sqcap x))\approx \\
         & \approx(x\sqcup z)\sqcap(z\sqcup0)\approx(x\sqcup z)\sqcap z\approx z, \\
p(x,z,z) & \approx(x\sqcup(z'\sqcap(z\sqcup z)))\sqcap(z\sqcup(z'\sqcap(z\sqcup x)))\approx(x\sqcup(z'\sqcap z))\sqcap(z\sqcup x)\approx \\
         & \approx(x\sqcup0)\sqcap(z\sqcup x)\approx x\sqcap(z\sqcup x)\approx x.
\end{align*}
\end{proof}

We are going to show also congruence regularity of the variety ${\mathcal Ls}$.

\begin{theorem}
Let $\mathcal V$ be a variety of bounded $\lambda$-lattices $(L,\sqcup,\sqcap,{}',0,1)$ with a unary operation satisfying the identities $0'\approx1$, $x\sqcap x'\approx0$ and {\rm(5)}. Then $\mathcal V$ is congruence regular. In particular, the variety ${\mathcal Ls}$ is congruence regular.
\end{theorem}

\begin{proof}
Put
\[
t(x,y):=(x'\sqcap(x\sqcup y))\sqcup(y'\sqcap(x\sqcup y)).
\]
Then
\[
t(x,x)\approx(x'\sqcap(x\sqcup x))\sqcup(x'\sqcap(x\sqcup x))\approx(x'\sqcap x)\sqcup(x'\sqcap x)\approx0\sqcup0\approx0,
\]
and if $t(x,y)=0$ then $x'\sqcap(x\sqcup y)=y'\sqcap(x\sqcup y)=0$ and hence
\[
x=x\sqcup0=x\sqcup(x'\sqcap(x\sqcup y))=x\sqcup y=y\sqcup x=y\sqcup(y'\sqcap(x\sqcup y))=y\sqcup0=y.
\]
If we put
\begin{align*}
t_1(x,y,z) & :=t(x,y)\sqcup z, \\
t_2(x,y,z) & :=(t(x,y))'\sqcap z
\end{align*}
then
\begin{align*}
t_1(x,x,z) & \approx t(x,x)\sqcup z\approx0\sqcup z\approx z, \\
t_2(x,x,z) & \approx(t(x,x))'\sqcap z\approx0'\sqcap z\approx1\sqcap z\approx z,
\end{align*}
and if $t_1(x,y,z)=t_2(x,y,z)=z$ then $t(x,y)\leq z\leq(t(x,y))'$ and hence $t(x,y)=t(x,y)\sqcap(t(x,y))'=0$ whence $x=y$. (Observe that in ${\mathcal Lc}$ we have $0'\approx0\sqcup0'\approx1$.)
\end{proof}

Finally, we are going to show the independence of identities (1) -- (4).

\begin{theorem}
\
\begin{enumerate}[{\rm(i)}]
\item Identities {\rm(1)} -- {\rm(4)} are independent,
\item identities {\rm(1)} -- {\rm(4)} do not imply identity {\rm(5)}.
\end{enumerate}
\end{theorem}

\begin{proof}
\
\begin{enumerate}[(i)]
\item The $\lambda$-lattice $(\{0,1\},\sqcup,\sqcap,0,1)$ with
\[
\begin{array}{c|cc}
x  & 0 & 1 \\
\hline
x' & 0 & 0
\end{array}
\]
satisfies (2), (3) and (4), but not (1) since
\[
(0\sqcup0)\sqcup(0'\sqcup0)=0\sqcup(0\sqcup0)=0\neq1.
\]
The $\lambda$-lattice $(\{0,1\},\sqcup,\sqcap,0,1)$ with
\[
\begin{array}{c|cc}
x  & 0 & 1 \\
\hline
x' & 1 & 1
\end{array}
\]
satisfies (1), (3) and (4), but not (2) since
\[
(1\sqcap1)\sqcap(1'\sqcap1)=1\sqcap(1\sqcap1)=1\neq0.
\]
The $\lambda$-lattice shown in Fig.~6

\vspace*{-8mm}

\[
\setlength{\unitlength}{7mm}
\begin{picture}(6,11)
\put(3,2){\circle*{.3}}
\put(1,4){\circle*{.3}}
\put(3,4){\circle*{.3}}
\put(1,6){\circle*{.3}}
\put(5,6){\circle*{.3}}
\put(1,8){\circle*{.3}}
\put(3,8){\circle*{.3}}
\put(3,10){\circle*{.3}}
\put(3,2){\line(-1,1)2}
\put(3,2){\line(0,1)8}
\put(3,2){\line(1,2)2}
\put(1,4){\line(0,1)4}
\put(1,6){\line(1,1)2}
\put(3,10){\line(-1,-1)2}
\put(3,10){\line(1,-2)2}
\put(2.875,1.25){$0$}
\put(.35,3.85){$a$}
\put(3.3,3.85){$b$}
\put(.35,5.85){$c$}
\put(5.3,5.85){$d$}
\put(.35,7.85){$e$}
\put(3.3,7.85){$f$}
\put(2.85,10.4){$1$}
\put(2.2,.3){{\rm Fig.~6}}
\end{picture}
\]

\vspace*{-3mm}

with $a\sqcup b=1$, $b\sqcup c=f$, $e\sqcap f=c$ and
\[
\begin{array}{c|cccccccc}
x  & 0 & a & b & c & d & e & f & 1 \\
\hline
x' & 1 & d & e & d & a & b & d & 0
\end{array}
\]
satisfies (1), (2) and (4), but not (3) since
\begin{align*}
(((a\sqcap b')\sqcup f)\sqcup(b\sqcup f))\sqcap((a\sqcap b')\sqcup b) & =(((a\sqcap e)\sqcup f)\sqcup f)\sqcap((a\sqcap e)\sqcup b)= \\
                                                                      & =((a\sqcup f)\sqcup f)\sqcap(a\sqcup b)=(f\sqcup f)\sqcap1= \\
																																			& =f\neq1=a\sqcup b=(a\sqcap e)\sqcup b= \\
																																			& =(a\sqcap b')\sqcup b.
\end{align*}																																			
The $\lambda$-lattice shown in Fig.~7

\vspace*{-8mm}

\[
\setlength{\unitlength}{7mm}
\begin{picture}(8,9)
\put(4,2){\circle*{.3}}
\put(1,4){\circle*{.3}}
\put(3,4){\circle*{.3}}
\put(5,4){\circle*{.3}}
\put(7,4){\circle*{.3}}
\put(2,6){\circle*{.3}}
\put(4,6){\circle*{.3}}
\put(6,6){\circle*{.3}}
\put(4,8){\circle*{.3}}
\put(4,2){\line(-3,2)3}
\put(4,2){\line(-1,2)2}
\put(4,2){\line(1,2)2}
\put(4,2){\line(3,2)3}
\put(1,4){\line(1,2)1}
\put(1,4){\line(3,2)3}
\put(3,4){\line(3,2)3}
\put(5,4){\line(-3,2)3}
\put(7,4){\line(-3,2)3}
\put(7,4){\line(-1,2)1}
\put(4,8){\line(-1,-1)2}
\put(4,8){\line(0,-1)2}
\put(4,8){\line(1,-1)2}
\put(3.875,1.25){$0$}
\put(.35,3.85){$a$}
\put(2.35,3.85){$b$}
\put(5.3,3.85){$c$}
\put(7.3,3.85){$d$}
\put(1.35,5.85){$e$}
\put(4.3,5.85){$f$}
\put(6.3,5.85){$g$}
\put(3.85,8.4){$1$}
\put(3.2,.3){{\rm Fig.~7}}
\end{picture}
\]

\vspace*{-3mm}

with $a\sqcup b=e$, $a\sqcup c=e$, $a\sqcup d=f$, $b\sqcup c=1$, $b\sqcup d=g$, $c\sqcup d=g$, $e\sqcap f=a$, $e\sqcap g=c$, $f\sqcap g=d$ and
\[
\begin{array}{c|ccccccccc}
x  & 0 & a & b & c & d & e & f & g & 1 \\
\hline
x' & 1 & g & f & f & e & d & b & a & 0
\end{array}
\]
satisfies (1), (2) and (3), but not (4) since
\begin{align*}
(((d\sqcap g)'\sqcap b)\sqcap(g\sqcap b))\sqcup((d\sqcap g)'\sqcap g) & =((d'\sqcap b)\sqcap b)\sqcup(d'\sqcap g)= \\
                                                                      & =((e\sqcap b)\sqcap b)\sqcup(e\sqcap g)=(b\sqcap b)\sqcup c= \\
																																			& =b\sqcup c=1\neq c=e\sqcap g=d'\sqcap g= \\
																																			& =(d\sqcap g)'\sqcap g.
\end{align*}
\item The $\lambda$-lattice shown in Fig.~8

\vspace*{-8mm}

\[
\setlength{\unitlength}{7mm}
\begin{picture}(6,9)
\put(3,2){\circle*{.3}}
\put(1,4){\circle*{.3}}
\put(5,5){\circle*{.3}}
\put(1,6){\circle*{.3}}
\put(3,8){\circle*{.3}}
\put(3,2){\line(-1,1)2}
\put(3,2){\line(2,3)2}
\put(1,4){\line(0,1)2}
\put(3,8){\line(-1,-1)2}
\put(3,8){\line(2,-3)2}
\put(2.875,1.25){$0$}
\put(.35,3.85){$a$}
\put(5.3,4.85){$b$}
\put(.35,5.85){$c$}
\put(2.85,8.4){$1$}
\put(2.2,.3){{\rm Fig.~8}}
\end{picture}
\]

\vspace*{-3mm}

with
\[
\begin{array}{c|ccccc}
x  & 0 & a & b & c & 1 \\
\hline
x' & 1 & b & a & b & 0
\end{array}
\]
satisfies (1) -- (4), but not (5) since
\[
a\sqcup(a'\sqcap(a\sqcup c))=a\sqcup(b\sqcap c)=a\sqcup0=a\neq c=a\sqcup c.
\]
\end{enumerate}
\end{proof}

Authors' addresses:

Ivan Chajda \\
Palack\'y University Olomouc \\
Faculty of Science \\
Department of Algebra and Geometry \\
17.\ listopadu 12 \\
771 46 Olomouc \\
Czech Republic \\
ivan.chajda@upol.cz

Miroslav Kola\v r\'ik \\
Palack\'y University Olomouc \\
Faculty of Science \\
Department of Computer Science \\
17.\ listopadu 12 \\
771 46 Olomouc \\
Czech Republic \\
miroslav.kolarik@upol.cz

Helmut L\"anger \\
TU Wien \\
Faculty of Mathematics and Geoinformation \\
Institute of Discrete Mathematics and Geometry \\
Wiedner Hauptstra\ss e 8-10 \\
1040 Vienna \\
Austria, and \\
Palack\'y University Olomouc \\
Faculty of Science \\
Department of Algebra and Geometry \\
17.\ listopadu 12 \\
771 46 Olomouc \\
Czech Republic \\
helmut.laenger@tuwien.ac.at
\end{document}